\documentclass[leqno, 12pt]{amsart}
\usepackage{amsmath,amsfonts,amsthm,amsopn,cite,mathrsfs}
\usepackage{amssymb,latexsym,amsmath,amscd,epsfig}

\theoremstyle{theorem}
\newtheorem{lemma}{Lemma}[section]
\newtheorem{theorem}{Theorem}[section]

\newtheorem*{fact}{Theorem I}{\bf}{\it}
{\bf}{\it}

\newtheorem{proposition}{Proposition}[section]

\numberwithin{equation}{section} \theoremstyle{remark}

\newtheorem*{remark}{Remark}

\newtheorem*{claim}{Claim}

\newcommand{\s}{{\mathbf S}}
\newcommand{\C}{{\mathbf C}}

\newcommand{\N}{{\mathbf N}}
\newcommand{\R}{{\mathbf R}}

\newcommand{\di}{{\rm div}}
\newcommand{\supp}{{\rm supp}}

\newcommand{\dist}{{\rm dist}}

\newcommand{\beq}{\begin{equation}}
\newcommand{\eeq}{\end{equation}}
\newcommand{\beqs}{\begin{equation*}}
\newcommand{\eeqs}{\end{equation*}}
\newcommand{\beg}{\begin{gather}}
\newcommand{\eeg}{\end{gather}}

\begin{document}

\title[Quantitative uniqueness]{Quantitative uniqueness for elliptic
equations with singular lower order terms}


\author{Eugenia Malinnikova}
\address{Department of Mathematics,
Norwegian University of Science and Technology,
7491, Trondheim, Norway}
\email{eugenia@math.ntnu.no}
\author{Sergio Vessella}
\address{Dipartimento di Matematica per le Decisioni, Universit\`{a}
degli Studi, 50134 Firenze, Italy}
\email{sergio.vessella@dmd.unifi.it}

\begin{abstract}
We use a Carleman type inequality of Koch and Tataru to obtain
quantitative estimates of unique continuation for solutions of
second order elliptic equations with singular lower order terms.
First we prove a three sphere inequality and then describe two
methods of propagation of smallness from sets of positive measure.
\end{abstract}

\maketitle

\section{Introduction}

In this work we deal with second-order uniformly elliptic
equations in a  bound\-ed domain $\Omega\subset\R^n$, $n\ge 3$. We
assume that the equation is in divergence form and terms of order
one and zero may have singularities. The conditions we impose on
the lower order terms imply the strong unique continuation
property, we refer the reader to the article of H.Koch and
D.Tataru, \cite{KT}, and references therein for the history of the
Carleman inequalities and strong unique continuation for
second-order elliptic equations.

\subsection{Problem of quantitative propagation of smallness}
Assume that a solution to a second-order uniformly elliptic
equation is bounded on the domain and is small on a subset of
positive measure, our aim is to estimate such a solution on an
arbitrary compact subset of the domain.  We refer to estimates of
this nature as quantitative propagation of smallness. Three sphere
inequalities for elliptic equations provide classical examples of
quantitative propagation of smallness; various versions of the
inequality can be found, for example, in \cite{La,G,A,B,K,ARRV}.
Our first result is a version of the three sphere inequality for
equations with singular coefficients in lower order terms, it is
derived from the inequality of H.Koch and D.Tataru.

We refer the reader to articles of N.Nadirashvili and S.Vessella,
\cite{N} and \cite{V}, in which the problem of propagation of
smallness for second-order elliptic equations from a set $E$ of
positive measure was considered. We mention also that similar
problems for the case of elliptic equations with analytic
coefficient were discussed in \cite{N1,V1,M}, methods used in these
works are of complex analytic nature.

In the second part of the article we give
two approaches to propagation of smallness for the case of singular
coefficients, both of them
use the three sphere inequality obtained in the first part of the
work. The first is an improvement of the one in \cite{V}. It uses
Carleman type inequality and gives estimates of $L^2$-norms, all
constants can be estimated explicitly. The second approach repeats
a clever argument of N.Nadirashvili, \cite{N}, we assume a
slightly better integrability of the lower order terms than for
the first approach, and estimate $L^\infty$-norms, the constants
are not explicit here but the asymptotic of the decay of solution
is better. The precise formulations of the results are given in
the next section.

\subsection{Formulation of the result}
We consider the equation 
\beq 
\label{eq:11} 
Pu=Vu+W_1\cdot\nabla
u+\nabla\cdot(W_2u), 
\eeq 
where $P=\di(g\nabla u)$,
$g(x)=\{g^{ij}(x)\}_{i,j=1}^{n}$ is a real-valued symmetric matrix
such that it satisfies, for a given constant $\lambda\in(0,1]$,
the uniform ellipticity condition in $\Omega$, 
\beq 
\label{eq:00}
\lambda|\zeta|^2\le
g(x)\zeta\cdot\zeta\le\lambda^{-1}|\zeta|^2,\quad x\in\Omega,
\zeta\in\R^n. 
\eeq 
We also assume that, for a given constant
$\Lambda_0>0$, the following Lipschitz condition holds 
\beq
\label{eq:02} 
|g(x)-g(y)|\le\Lambda_0|x-y|,\quad x,y\in\Omega.
\eeq 
Finally, the lower order terms are assumed to satisfy the
following integrability conditions: 
\beq 
\label{eq:01} V\in
L^{n/2}(\Omega)\quad {\rm{and}}\quad  W_1,W_2\in L^s(\Omega)\
{\rm{with}}\ s>n, 
\eeq 
here $W_1, W_2:\Omega\rightarrow \R^n$ and
$V:\Omega\rightarrow\R$.



The main aim of the work is to obtain quantitative propagation of
smallness from sets of positive measure for solutions of (\ref{eq:11}).
The problem setting is the following:

\textsl{Let $E,K$ be compact subsets of $\Omega$ and let $E$ have
positive measure. Find a function $\phi(\epsilon)$,
$\lim_{\epsilon\rightarrow 0}\phi(\epsilon)=0$, such that any
solution $u$ of $(\ref{eq:11})$ that satisfies 
\beq 
\label{eq:53}
\|u\|_{L^{2}(\Omega)}\le 1,\quad \|u\|_{L^2(E)}\le
\epsilon\ 
\eeq 
is bounded in $L^2(K)$ by
\[
\|u\|_{L^2(K)}\le\phi(\epsilon).\] } The existence of such a
function $\phi$ can be proved in the following way. Assume that
there is a sequence $\{u_j\}$ of solutions to (\ref{eq:11}) such
that $\|u_j\|_{L_2(\Omega)}\le |\Omega|^{1/2}$,
$\|u_j\|_{L^2(E)}\le \epsilon_j$, where $\epsilon_j\rightarrow 0$
and $\|u_j\|_{L^2(K)}\ge  c>0$ for each $j$. Applying the
Caccioppoli inequality (see below), we obtain that $\{u_j\}$ is
bounded in $W_2^1(\Omega')$, where $\Omega'\subset\subset\Omega$.
By choosing a  subsequence $\{u_{j_l}\}$, we find a solution $u$
to (\ref{eq:11}) in $\Omega'$ such that $\{u_{j_l}\}$ weakly
converges to $u$ in $W_2^1(\Omega')$ and in $L^2(\Omega)$. Then
$u=0$ on $E$ while $\|u\|_{L^2(K)}\ge c>0$. According to a result
of R.Regbaoui \cite{R}, if $u$ vanishes on a set of positive
measure then $u$, in particular, has a zero of infinite order at
some point and by the strong unique continuation property proved
by Koch and Tataru \cite{KT}, $u\equiv 0$.

In this work we describe constructive schemes that provide
quantitative estimates of $\phi$. We remark also
that in our schemes $\phi$ does not depend on $E$ but only on $K$,
the measure of $E$, and the distance from $E$ to the boundary of
$\Omega$.

Let  $\Omega(\rho)=\{x\in\Omega:\dist\{x,\partial\Omega)>4\rho\}$
for each $\rho>0$. Since in what follows we shall assume that
$\Omega$ is a bounded connected open set with Lipschitz boundary,
we may consider only $\rho<\rho^*$ such that $\Omega(\rho)$ is
also connected. Our main results are the following:

\begin{theorem}
\label{th:m1} Let $\Omega$ be a bounded domain with Lipschitz
boundary, $u\in W^1_2(\Omega)$ be a solution of (\ref{eq:11}), and
the coefficients of the equation satisfy
(\ref{eq:00}-\ref{eq:01}). Further, let $\rho<\rho^*$ and let $E$
be a measurable subset of $\Omega(\rho)$ of positive measure such
that (\ref{eq:53}) holds. Then \beq \label{eq:50}
\|u\|_{L^2(\Omega(\rho))}\le C|\log\epsilon|^{-c} , \eeq where $C$
and $c$ depend on $\Omega$, $\lambda$, $\Lambda_0$, $V$, $W_1$,
$W_2$, $|E|$, and $\rho$ only.
\end{theorem}

\begin{theorem}
\label{th:m2} Let $\Omega$ be a bounded domain with Lipschitz
boundary, $u\in W^1_2(\Omega)$ be a solution of (\ref{eq:11}),
where $P$ satisfies (\ref{eq:00}-\ref{eq:02}) and \beq
\label{eq:01n} V\in L^{s/2}(\Omega)\quad {\text{and}}\quad
W_1,W_2\in L^s(\Omega)\ {\text{with}}\ s>n, \eeq
 and let $E$ be a measurable subset of $\Omega(\rho)$, $\rho<\rho^*$, of positive measure such that
 \beq
\label{eq:53n} \|u\|_{L^{\infty}(\Omega)}\le 1,\quad
\|u\|_{L^\infty(E)}\le \epsilon\ \eeq holds. Then \beq
\label{eq:50n} \|u\|_{L^\infty(\Omega(\rho))}\le
C\exp(-c(|\log\epsilon|)^\mu), \eeq where $c,C$ and $\mu$ depend
on $\Omega$, $\lambda$, $\Lambda_0$, $V$, $W_1$, $W_2$, $|E|$, and
 $\rho$ only.
\end{theorem}

\subsection{The structure of the article}
Preliminary results are collected in the next section, we
formulate a version of the Carleman inequality due to Koch and
Tataru that implies strong unique continuation for equations we
consider; we also prove the Caccioppoli inequality for solutions
of this equations. In Section 3 we  obtain the doubling property
for solutions of our equations and prove a three sphere
inequality. The proof of Theorem \ref{th:m1} appears at the end of
Section 4. First in this section we show that the Caccioppoli
inequality for solution $u$ and the doubling property yield
Muckenhoupt condition for the weight $|u|^2$, then we apply the
three sphere inequality.  Finally, in Section 5 we reproduce (a
slightly modified) argument of Nadirashvili that, in combination
with the three sphere inequality for the class of equations we
consider, gives a proof of Theorem \ref{th:m2}.

\section{Preliminaries}
In this section we introduce the notation and formulate the results needed in the sequel.

\subsection{Notation and an inequality of Koch and Tataru}
We work with standard functional spaces $L^s(\Omega)$ and
$W^{l}_m(\Omega)$ and always assume that $\Omega$ is a bounded
domain; our results reflect local properties of functions inside
the domain, so without loss of generality we consider only domains
with Lipschitz boundary. Solutions of (\ref{eq:11}) are defined as
weak solutions in $W_2^1(\Omega)$, standard definitions and
notation that we use can be found in \cite{LU}.

Our basic tool is the following version of Carleman inequality.
\begin{fact}
 Assume that coefficients of (\ref{eq:11}) satisfy (\ref{eq:00}-\ref{eq:01})
Then there exists $H_0$, $r$, and $\tau_0$ such that for every
$\tau>\tau_0$ and each function $v$ vanishing at $x_0\in\Omega$,
$\dist(x_0,\bar\Omega)>2r$, and with $\supp v\subset B(x_0,r)$
that solves
\[
Pv-Vv-W_1\cdot\nabla v-\nabla\cdot(W_2v)=f\] there exists $\phi$
such that \beq \label{eq:12} \tau\le -r\partial_r\phi\le
H_0\tau,\quad |\partial_\theta\phi|\le|r\partial_r\phi| \eeq and
the following Carleman estimate holds \beq \label{eq:10}
\|e^{\phi(x)}v\|_{L^p}\le a_n\|e^{\phi(x)}f\|_{L^q},\eeq where
$q=\frac{2n}{n+2}$, $p=\frac{2n}{n-2}$ and $a_n$ depends only on
the dimension $n$ of the space.
\end{fact}

The proof of this theorem repeats that of Corollary 3.3 in \cite{KT}.
We use stronger assumptions on the gradient terms and the matrix $g$.
First we note that, after a simple change of the coordinates, metric $g^*$ satisfies $g^*(x_0)=I_n$ and we have
\[|g^*(x)-I_n|+|x-x_0||\nabla g*(x)|\le C(\lambda,\Lambda_0)|x-x_0|.\]
Then we consider the construction of function $h$ in Sections 6-7
of \cite{KT}. We claim that for our assumptions on the
coefficients one can choose $a_j=q^j$, where $q=q(s)<1$. Indeed,
then both inequalities (6.5) and (7.1) in \cite{KT} are satisfied (provided that $q$ is
close enough to 1). Then function $h$ in Lemma 6.1 of \cite{KT}
satisfies $h'(s)\in[\tau, H_0\tau]$ for some $H_0=H_0(q)$.

\subsection{Caccioppoli's inequality}
The Caccioppoli inequality holds for solutions of elliptic
equations that we consider and will be used several times in our
calculations. We give a proof for the convenience of the reader.

\begin{proposition} \label{pr:0}Let $R\subset\subset \tilde{R}\subset\Omega$, assume that coefficients of equation (\ref{eq:11})
satisfy the conditions (\ref{eq:00}-\ref{eq:01}) and
\[\|V\|_{L^{n/2}(\tilde{R})}<\varepsilon,\quad\|W_1\|_{L^n(\tilde{R})}+\|W_2\|_{L^n(\tilde{R})}<\varepsilon.\]
Then there exist $\varepsilon_0$ and $C_1$, depending on $ R,
\tilde{R}$, and $\lambda$ only, such that if
$0<\varepsilon<\varepsilon_0$ then the following inequality holds
\beq\label{eq:15} \|\nabla u\|_{L^2(R)}\le
C_1\|u\|_{L^2(\tilde{R})} \eeq for any solution $u$  to
(\ref{eq:11}).
\end{proposition}

\begin{proof}
Let $\eta\in C^\infty_0(\tilde{R})$, $\eta=1$ on $R$. Then
\begin{multline*}
\|\eta\nabla  u\|^2_{L^2(\tilde{R})}\le \lambda^{-1}\int_{\tilde{R}}\eta^2 g\nabla  u\cdot\nabla u \le\\
\left| \lambda^{-1}\int_{\tilde{R}} (g\nabla u)\cdot\nabla(\eta^2
u)\right|+ \left|\lambda^{-1}\int_{\tilde{R}} (g\nabla u)\cdot
(2u\eta\nabla \eta)\right|.
\end{multline*}
By the Cauchy inequality, the last term admits an estimate
\[
\lambda^{-1}\left|\int_{\tilde{R}} (g\nabla u)\cdot (2u\eta\nabla
\eta)\right|\le 2\lambda^{-2}\int_{\tilde{R}}|\eta \nabla
u||u\nabla\eta|\le \frac12 \|\eta\nabla
u\|_2^2+2\lambda^{-4}\|u\nabla\eta\|_2^2.\] Thus, using
(\ref{eq:11}), we obtain
\begin{multline*}
\|\eta\nabla u\|^2_{L^2(\tilde{R})}\le
2\lambda^{-1}\left|\int_{\tilde{R}}(g\nabla u)\cdot \nabla (\eta^2 u)\right|+4\lambda^{-4}\|u\nabla\eta\|^2_2\\
\le
2\lambda^{-1}\left(\int_{\tilde{R}}(|V||\eta u|^2+|W_1||\eta\nabla u||\eta u|+ |W_2||\nabla(\eta^2u)||u|)\right)+\\
4\lambda^{-4}\|u\nabla\eta\|^2_{L^2(\tilde{R})}.
\end{multline*}
Next, we apply H\"{o}lder's inequality
\begin{multline*}
\|\eta\nabla u\|^2_{L^2(\tilde{R})}\le  2\lambda^{-1}(\| V\|_{L^{n/2}(\tilde{R})}\|\eta u\|_{L^p}^2+\\
\||W_1|+|W_2|\|_{L^n(\tilde{R})}\|\eta\nabla u\|_{L^2}\|\eta
u\|_{L^p}+
\|W_2\|_{L^n(\tilde{R})}\|\eta u\|_{L^p}\|u\nabla\eta\|_{L^2})\\
+4\lambda^{-4}\|\nabla\eta\|^2_{L^\infty}\|u\|^2_{L^2(\tilde{R})}.
\end{multline*}
Finally by Sobolev's embedding inequality, see for example
\cite[page 74]{LU},
\begin{multline*}
\|\eta\nabla u\|^2_{L^2(\tilde{R})}
\le  2\lambda^{-1}C\| V\|_{L^{n/2}(\tilde{R})}(\|\eta \nabla u\|^2_{L^2}+\|\nabla \eta\|^2_{L^\infty}\|u \|^2_{L^2(\tilde{R})})\\
+2\lambda^{-1}C\||W_1|+|W_2|\|_{L^n(\tilde{R})}(\|\eta\nabla u\|_{L^2}^2+\|\nabla \eta\|^2_{L^\infty}
\|u(x)\|^2_{L^2(\tilde{R})})+\\
4\lambda^{-4}\|\nabla\eta\|^2_{L^\infty}\|u\|^2_{L^2(\tilde{R})}.
\end{multline*}
Taking $\varepsilon$ small enough, we absorb all the terms with
$\|\eta\nabla u\|_{L^2}$ in the left hand side of the inequality
and obtain (\ref{eq:15}).
\end{proof}

\begin{remark}
It follows from the calculations above that (\ref{eq:15}) holds
with  \[C_1=C_1(\lambda, \Omega, V, W_1,W_2)\|\nabla
\eta\|_{L^{\infty}}.\] In particular, we will apply the inequality
to two concentric balls and then \beq\label{eq:16} \|\nabla
u\|_{L^2(B_{ar}(x))}\le \tilde{C}_1r^{-1}\|u\|_{L^2(B_r(x))}, \eeq
where $a<1$, $\tilde{C}_1$ depends on $\lambda, \Omega, V, W_1,
W_2$ and on $a$ only.
\end{remark}


\section{Doubling property and three sphere inequality}
Using inequalities from the previous section, we prove the doubling property and the three sphere theorem for solutions of elliptic equations with singular lower order terms.

\subsection{Doubling inequality}
First we obtain a doubling inequality for solutions of
(\ref{eq:11}).
\begin{proposition}
\label{pr:1} Suppose that $P,V,W_1,$ and $W_2$ satisfy the
conditions (\ref{eq:00}-\ref{eq:01}).  Then there exists
$\rho_0>0$ and $\kappa<\frac{1}{4}$ such that for any solution
$u\in W^{1}_2(\Omega)$ of (\ref{eq:11}), any $\rho<\rho_0$, any
$\tilde{x}\in\Omega(\rho)$,  and  $r<\kappa\rho$, we have 
\beq
\label{eq:20} 
\int_{B_{2r}(\tilde{x})}|u|^2\le
C(u,\Omega,\rho)\int_{B_{r}(\tilde{x})}|u|^2, \eeq where
$r<\kappa\rho$ and \beq \label{eq:19} C(u,\Omega,\rho)=
C_0\max_{x\in\Omega(\rho/2)}\left( \frac{\left\Vert u\right\Vert
_{L^{2}(B_{\rho }(x))}}{\left\Vert u\right\Vert _{L^{2}(B_{2\kappa
\rho }(x))}}\right) ^{H_{1}}, 
\eeq  
$H_{1}=5H_{0}$ ($H_{0}$
is defined in Theorem I) and $C_0$
 depends only on
$\lambda,\Lambda_0,\Omega,V,W_1,$ and $W_2$.
\end{proposition}

The inequality above shows that doubling constants for small
scales are controlled by doubling constants on some fixed scale.
General discussions of the doubling property for solutions of
elliptic equations and the frequency of solutions were initiated
by the works of N.Garofalo and F.-H.Lin, see \cite{GL1,Li}. We
will derive the doubling inequality from Theorem I.
Similar result was proved recently by B.Su in
\cite{BS}, where calculations are performed for the case
$P=\Delta,\quad V=W_2=0$;  the author also mentions that the
general case of equation (\ref{eq:11}) can be treated similarly.
We consider the general case and obtain doubling for $L^2$-norms;
we also write down the precise expression for $C(u,\Omega,\rho)$ in the doubling
inequality since we will need it for propagation of smallness
estimates.

\subsection{Localization of the problem}
 
Let $0<\varepsilon<\varepsilon_0$, where $\varepsilon_0$ is
defined in Proposition \ref{pr:0}. Let $\rho>0$ and let
$\alpha_\rho$ be a molifier such that
$\alpha_\rho(x)=\alpha(\rho^{-1} x)$, where $\alpha$ is a radial
function, $\alpha(y)=1$ when $|y|<1$, and $\supp\,\alpha\subset
B_2(0)$. We fix $\tilde{x}\in\Omega(\rho)$ and define
\[\tilde{g}(x)=g(\tilde{x}+x)\alpha_\rho(x)+g(\tilde{x})(1-\alpha_\rho(x)).\]
Then $\tilde{g}\in W^{1,\infty}(\R^n)$, moreover \beq
\label{eq:13} 
\rho\|\nabla
g\|_{L^\infty}+\|g({\cdot})-g(\tilde{x})\|_{L^\infty(B_{2\rho}(\tilde{x}))}<\varepsilon,
\eeq when $\rho$ is small enough. Further we define
$\tilde{Y}(x)=Y(\tilde{x}+x)\alpha_\rho(x)$, where $Y=V,W_1,W_2$.
Once again, when $\rho$ is sufficiently small we achieve \beq
\label{eq:14} \|\tilde{V}\|_{L^{n/2}}\le\varepsilon,\quad
\|\tilde{W}_1,\tilde{W}_2\|_{L^s}\le \varepsilon. \eeq At
this point we choose $\rho_0$ small enough and always assume that
$\rho<\rho_0$, then (\ref{eq:13}-\ref{eq:14}) hold.

Now let $r<\kappa\rho$ (we will choose $\kappa<1/4$ later) and let
$\eta\in C_0^\infty(B_\rho(0))$ be such that $\eta(x)=\eta(|x|)$,
$0\le\eta\le 1$ and
\[\eta=0\quad {\rm{for}}\ |x|<\frac{r}{2} \ {\rm{and}}\
|x|>\frac{3{\rho}}{4},\]
\[\eta=1\quad {\rm{for}}\ \frac{3r}{4}<|x|< \frac{{\rho}}{2},\]
\[|\nabla\eta|\le \frac{A}{r},\quad |D^2\eta|\le\frac{A^2}{r^2}\ {\rm{in}}\ B_{3r/4 }\setminus B_{r/2},\]
\[|\nabla\eta|\le \frac{A}{\rho},\quad |D^2\eta|\le \frac{A^2}{\rho^2}\ {\rm{in}}\ B_{3{\rho}/4}\setminus B_{{\rho}/2}.\]

Let $u\in W^1_2(\Omega)$ be a solution of (\ref{eq:11}) and let
$v(x)=\eta(x) u(\tilde{x}+x)$. We consider
\[ L(v)=\di(\tilde{g}\nabla v)-(\tilde{V}v+\tilde{W}_1\cdot\nabla v+\nabla\cdot(\tilde{W}_2v)).\]
Using (\ref{eq:11}), we note that \beq \label{eq:17} L(v)=2
(\tilde{g}\nabla\eta)\cdot\nabla u(\cdot+\tilde{x})-
u(\cdot+\tilde{x})\nabla\eta\cdot(\tilde{W}_1+\tilde{W}_2)+
u(\cdot+\tilde{x})\tilde{P}\eta, \eeq where
$\tilde{P}f=\di(\tilde{g}\nabla f)$; note that the right-hand side
of the last equation is a function in $L^2(\Omega)$ with a compact
support.

We apply the inequality (\ref{eq:10}) to the function $v$ that
vanishes at the origin and infinity and the operator $L$.
By Theorem I for every $\tau>0$ there
exists $\phi=\phi(\tau,v,\tilde{W}_1,\tilde{W}_2)$ such that
(\ref{eq:12}) is satisfied and \beq \label{eq:10a}
\|e^{\phi(x)}v\|_{L^p}\le a_n\|e^{\phi(x)}(Lv)\|_{L^q}.\eeq

It is a simple, but tedious, matter to check that relations
(\ref{eq:12}) imply that for some $b_0$, depending on $n$ only,
we have \beq \label{eq:22} \min_{|x|\leq
d}e^{\phi(x)}\ge\max_{|x|\geq b_0d}e^{\phi(x)}\quad{\rm{ and\
then}}\quad \min_{|x|\leq d}e^{\phi(x)}\ge
e^{\beta\tau}\max_{|x|\geq bd}e^{\phi(x)} \eeq for any $b\geq b_0$
and $\beta=\ln b-\ln b_0$.

\subsection{Doubling from Carleman's estimate}
We rewrite (\ref{eq:10a}) taking into account (\ref{eq:17}),
\begin{multline}
\label{eq:21}
 \|e^{\phi(x)}v\|_{L^p}\le
  2a_n(
\|e^{\phi(x)} \tilde{g}\nabla\eta\cdot\nabla
u(x+\tilde{x})\|_{L^q}
+\|e^{\phi(x)}u(x+\tilde{x})\tilde{P}\eta\|_{L^q}\\
+\|e^{\phi(x)}(\nabla\eta(x))u(x+\tilde{x})(\tilde{W}_1+\tilde{W}_2)\|_{L^q})=2a_n(S_1+S_2+S_3).
\end{multline}
Let us denote, up to the end of the present section, by  $B_t$ the
ball with center at $\tilde{x}$ and radius $t$. The first term in
the right hand side of (\ref{eq:21}) has the following estimate
\begin{multline}
\nonumber S_1=\|e^{\phi(x)} \tilde{g}\nabla\eta\cdot\nabla
u(x+\tilde{x})\|_{L^q}\le
\lambda^{-1}\|e^{\phi(x)}|\nabla \eta||\nabla u(\cdot+\tilde{x})|\|_{L^q}\leq\\
\lambda^{-1} A\left(\rho^{-1}\|\nabla
u\|_{L^q(B_{3\rho/4})}\max_{|x|\geq \rho/2}e^{\phi(x)}+
r^{-1}\|\nabla u||_{L^q(B_{3r/4})} \max_{|x|\geq
r/2}e^{\phi(x)}\right).
\end{multline}
Now we apply the  H\"{o}lder inequality and the Caccioppoli
inequality, (\ref{eq:16}),
\[
\|\nabla u\|_{L^q(B_{3t/4})}\le c_nt\|\nabla
u\|_{L^2(B_{3t/4})}\le C_2\|u\|_{L^2(B_{t})},\] where $C_2$
depends on $\lambda, V, W_1$ and $W_2$ only. Then
\[
S_1\le  C_3
\left(\rho^{-1}\|u\|_{L^2(B_{5\rho/6})}\max_{|x|\geq\rho/2}e^{\phi(x)}+
r^{-1}\|u\|_{L^2(B_{r})} \max_{|x|\geq r/2}e^{\phi(x)}\right).
\]
where $C_3$ depends on $\lambda, V, W_1$ and $W_2$ only. \\The
next term is bounded by
\begin{multline}
\nonumber S_2=\|e^{\phi(x)}u(x+\tilde{x})\tilde{P}\eta\|_{L^q} \le
 \|e^{\phi(x)}|u(x+\tilde{x})|(\lambda^{-1}|D^2\eta|+\Lambda_0|\nabla\eta|)\|_{L^q}\le\\
C_4\left(\rho^{-1}\|u\|_{L^2(B_{3\rho/4})}\max_{|x|\geq\rho/2}e^{\phi(x)}+r^{-1}
\|u\|_{L^2(B_{r})} \max_{|x|\geq r/2}e^{\phi(x)}\right),
\end{multline}
where $C_4$ depends only on $\lambda,\Lambda_0, \Omega, V, W_1$
and $W_2$. Finally, for the last term we have
\begin{multline*}
S_3=\|e^{\phi(x)}(\nabla\eta(x))u(x+\tilde{x})(\tilde{W}_1+\tilde{W}_2)\|_{L^q}\le\\
A(\|\tilde{W}_1\|_{L^n}+\|\tilde{W}_2\|_{L^n})(\rho^{-1}\|u\|_{L^2(B_{3\rho/4})}\max_{|x|\geq
\rho/2}e^{\phi(x)}+r^{-1} \|u\|_{L^2(B_{r})} \max_{|x|\geq
r/2}e^{\phi(x)}).
\end{multline*}
Thus the inequality (\ref{eq:21}) becomes
\begin{multline}
\label{eq:23} \|e^{\phi(x)}\eta(x) u(x)\|_{L^p}\le\\
C_5(\rho^{-1}\|u\|_{L^2(B_{5\rho/6})}\max_{|x|\geq\rho/2}e^{\phi(x)}+r^{-1}
\|u\|_{L^2(B_{r})} \max_{|x|\geq r/2}e^{\phi(x)}).
\end{multline}
On the other hand for any $\delta\in(0,1)$
\[
\|e^{\phi(x)}\eta(x) u(x)\|_{L^p}\ge \frac 12
\min_{|x|\leq\delta\rho}e^{\phi(x)}\|\eta
u\|_{L^p(B_{\delta\rho})}+ \frac 12 \min_{|x|\leq 2r}e^{\phi(x)}\|
u\|_{L^p(B_{2r}\setminus B_r)}.
\]
Clearly, by  H\"{o}lder's inequality, we have for any function
$\psi$
\[
\|\psi\|_{L^p(B_{2t}\setminus B_t)}\ge
\frac{c_n}{t}\|\psi\|_{L^2(B_{2t}\setminus B_{t})}.\]
If we assume
that $\delta\rho\in(2r,\rho/2)$ and combine the last three
estimates, we get
\begin{multline}
\label{eq:24}
\frac{c_n}{2r} \min_{|x|\leq 2r}e^{\phi(x)} \|u\|_{L^2(B_{2r}\setminus B_r)}+
\frac {c_n}{ 2\delta\rho} \min_{|x|\leq\delta\rho}e^{\phi(x)}\|u\|_{L^2(B_{\delta \rho}\setminus B_{\delta\rho/2})}\le\\
C_5\left(
\rho^{-1}\|u\|_{L^2(B_{5\rho/6})}\max_{|x|\geq\rho/2}e^{\phi(x)}+r^{-1}
\|u\|_{L^2(B_{r})} \max_{|x|\geq r/2}e^{\phi(x)}\right).
\end{multline}

Now we fix $\rho<\rho_0$ and choose
$\delta<\min\{(2eb_0)^{-1},c_n(2C_5)^{-1}, 1/9\} $ (see
(\ref{eq:22}) and (\ref{eq:24})); define $\kappa=\delta/8$ and
\beq \label{eq:29}
\tau=\max_{x\in\overline{\Omega(\rho/2)}}\log\frac{\|u\|_{L^2(B_{\rho}(x))}}{\|u\|_{L^2(B_{2\kappa\rho}(x))}}.
\eeq
We assume that $\phi$ corresponds to this $\tau$ (see Theorem I) and continue the estimates. We have
\[\| u\|_{L^2(B_{\delta \rho}\setminus
B_{\delta\rho/2})}\ge \|u\|_{L^2(B_{\kappa\rho}(x))},
\]
where $|x-\tilde{x}|=6\kappa\rho$. Clearly,
$x\in\overline{\Omega(\rho/2)}$ and the definition of $\tau$ gives
\[
\frac{c_n}{2\delta}e^{\tau}\|u\|_{L^2(B_{\delta
\rho}\setminus B_{\delta\rho/2})}\ge
C_5\|u\|_{L^2(B_{\rho}(x))}\ge C_5\|u\|_{L^2(B_{5\rho/6})}, \]
since $\rho-6\kappa\rho>5\rho/6$. This yields
\[\frac {c_n}{ 2\delta\rho} \min_{|x|\leq\delta \rho}
e^{\phi(x)}\| u\|_{L^2(B_{\delta\rho}\setminus B_{\delta\rho/2})}\ge C_5 \rho^{-1}\max_{|x|\geq b_0\delta\rho}e^{\phi(x)}\|u\|_{L^2(B_{5\rho/6})}.
\]
We remark that $b_0\delta\rho<\rho/2$ and then (\ref{eq:24})
implies
\[
\min_{|x|\leq 2r}e^{\phi(x)}\|u\|_{L^2(B_{2r})}\le C_6
\max_{|x|\geq r/2}e^{\phi(x)}\|u\|_{L^2(B_{r})}.
\]
Thus we obtain
\[
\|u\|_{L^2(B_{2r})}\le C_6 e^{{5H_{0}}\tau}\|u\|_{L^2(B_{r})},\]
for $\tau=\tau(u)$ defined by (\ref{eq:29}). Proposition
\ref{pr:1} is proved.

\subsection{Three sphere inequality}
\label{s:th} Our next result is a version of three sphere
inequality for equations with singular coefficients, once again we use Theorem I.
\begin{theorem}
\label{pr:3} Let $u$ be a solution of (\ref{eq:11}) in $\Omega$,
we assume that (\ref{eq:00}-\ref{eq:01}) holds. Let
$2r<R<\rho<\rho_0$, $R<\kappa\rho$, and $x\in\Omega(\rho)$
($\rho_0$ and $\kappa$ are as in Proposition \ref{pr:1}). Then the
following inequality holds
\beq
\label{eq:N1} \ R^{-1}\|u\|_{L^2(B_R(x))} \le\\C_{7}M^\alpha\sigma^{1-\alpha},
\eeq
where $C_7$ depends on the coefficients of the differential
equation but does not depend on $u$,
\[M=\rho^{-1}\|u\|_{L^2(B_\rho(x))},\quad \quad
\sigma=r^{-1}\|u\|_{L^2(B_r(x))},\]
and
\[\alpha=\frac{2{H_{0}}\log\frac{2R}{r}}{2{H_{0}}\log\frac{2R}{r}+\log\frac{\rho}{2b_0R}}.\]
\end{theorem}
\begin{proof}
The proof of inequality (\ref{eq:N1}) follows by standard arguments.
Here we give only a sketch of such a proof for the reader convenience.  For each $\tau>\tau_0$ (where $\tau_0$ is defined in Theorem I)
we apply the inequality of Theorem I to
$v=\eta u$ in a small enough ball and repeat the estimates above.
Inequality (\ref{eq:23})  implies
\begin{multline*}
R^{-1}\|u\|_{L^2(B_R\setminus B_{r})}\min_{|x|\leq R}e^{\phi(x)}\le\\
C_8(\rho^{-1}\|u\|_{L^2(B_\rho)}\max_{|x|\geq\rho/2}e^{\phi(x)}+r^{-1}
\|u\|_{L^2(B_{r})} \max_{|x|\geq r/2}e^{\phi(x)}).
\end{multline*}
We add $R^{-1}\|u\|_{L^2(B_{r})}$ to both sides and, using the
properties of $\phi$, see (\ref{eq:22}), we obtain
\[
R^{-1}\|u\|_{L^2(B_R)}\le
C_{9}(e^{-\tau\log\frac{\rho}{2b_0R}}\rho^{-1}\|u\|_{L^2(B_\rho)}+
e^{2{H_{0}}\tau\log\frac{2R}{r}}r^{-1}\|u\|_{L^2(B_r)}).\] Now,
denote $\log\frac{\rho}{2b_0R}=m$ and $2{H_{0}}\log\frac{2R}{r}=l$,
we get
\beq
\label{eq:NN2} R^{-1}\|u\|_{L^2(B_R)}\le
C_9(e^{-\tau m  }M+e^{\tau l}\sigma).\
\eeq
Now, let us denote
\[\tau_1=\frac{\log M-\log \sigma}{l+m}.\]
Notice that $e^{-\tau_1 m}M=e^{\tau_1 l}\sigma$. If $\tau_1>\tau_0$,
then we choose $\tau=\tau_1$ in (\ref{eq:NN2})and we get
\[
R^{-1}\|u\|_{L^2(B_R)}\le
2C_{9}M^\alpha\sigma^{1-\alpha}.\]
Otherwise, if $\tau_1\leq\tau_0$ inequality (\ref{eq:N1}) follows trivially.
\end{proof}

\begin{remark}
We note that if $s>n$ and $V\in L^{s/2}(\Omega)$, $W_1, W_2\in L^s(\Omega)$ then the constants in
(\ref{eq:19}) and (\ref{eq:N1}) may be chosen to depend only  on $\lambda, \Lambda_0, \Omega, s,$ and $\|V\|_{L^{s/2}(\Omega)},
\|W_1\|_{L^s(\Omega)}, \|W_2\|_{L^s(\Omega)}$.
\end{remark}

\section{Propagation of smallness}
In this section we prove the main result formulated in the
introduction. First we show how to prolongate the smallness of a
solution to a certain ball and then we apply Theorem \ref{pr:3}.
By normalization we may always assume that $|\Omega|\le
1$.

\subsection{From Doubling and Caccioppoli to Reverse H\"{o}lder and Muckenhoupt}
Throughout this section $u$ is a solution of (\ref{eq:11}) and
$C(u,\Omega,\rho)$ is given by (\ref{eq:19}). The Caccioppoli
inquality (\ref{eq:16}) and the doubling inequality imply
\[
\|\nabla u\|_{L^2(B_r(x))}\le
\tilde{C}_1r^{-1}\|u\|_{L^2(B_{2r}(x))}\le \tilde{C}_1 r^{-1} C(u,
\Omega, \rho)\|u\|_{L^2(B_{r}(x))},
\]
for $r<\kappa\rho$ and $\dist(x,\partial\Omega)>4\rho$,
$\rho<\rho_0$. Then by Sobolev's embedding theorem
\[
\|u\|_{L^{p}(B_{r}(x))}\le \tilde{C}_1
r^{-1}C(u,\Omega,\rho)\|u\|_{L^2(B_{r}(x)  )},\] where
$p=\frac{2n}{n-2}>2$, so we obtain the Reverse H\"{o}lder
inequality, the constant is uniform for $x\in\Omega(\rho)$
provided that $r<\kappa\rho.$

It is well known that Reverse H\"{o}lder's inequality implies the
Muckenhoupt condition for the weight $|u|^2$. Let $F$ be a
measurable subset of a ball $B:=B_r(x)$, where $r<\kappa\rho$ and
$x\in\Omega(\rho)$. We apply the H\"{o}lder inequality and then
the Reverse H\"{o}lder inequality obtained above
\[
\left(\int_{F}|u|^2\right)^{1/2}\le
\left(\int_B|u|^p\right)^{1/p}|F|^{1/n}\le
C_*\left(\int_B|u|^2\right)^{1/2}\left(\frac{|F|}{|B|}\right)^{1/n},
\]
where $C_*=C_{10}C(u,\Omega,\rho)$, once again $C_{10}$ depends only on
$\lambda,$ $\Lambda_0,$ $\Omega,$ $V,$ $W_1, W_2$.

Assume now that $G\subset B$ and
\[
|G|>(1-\alpha)|B|,\] where
\[\alpha=\alpha(u,\Omega,\rho)=2^{-n/2}C_{10}^{-n}C(u,\Omega,\rho)^{-n}.\]
Then  applying the last inequality to $F=B\setminus G$ we get
\beq\label{eq:28} \int_G|u|^2= \int_B|u|^2-\int_F|u|^2\ge
\int_B|u|^2\left(1-C_*^2\alpha^{2/n}\right)\ge
\frac1{2}\int_B|u|^2. \eeq


\subsection{Lemma of Nadirashvili}
Let a cube $Q_0$ be fixed.
We consider all dyadic sub-cubes of $Q_0$. First, $Q_0$ is divided into $2^n$ sub-cubes with the length of the side one half of that of $Q_0$, we denote them $Q^{(l)}, 1\le l\le 2^n,$ and call cubes of rank one.
Each cube $Q^{(l_1,...,l_r)}$ of rank $r$ is divided into $2^n$ sub-cubes of rank $r+1$ that are denoted by $Q^{(l_1,..., l_r, l_{r+1})}$,
$1\le l_{r+1}\le 2^n$. We will also say that  $Q^{(l_1,...,l_r)}$ is the dyadic parent of $Q^{(l_1,..., l_r, l_{r+1})}$. The dyadic parent of $Q_0$ is $Q_0$ itself.

We will use the following statement that can be found in \cite{N1}, the proof is included for the convenience of the reader.

\begin{lemma}
\label{l:squares}
Let $\mathcal{F}$ be a finite family of disjoint dyadic cubes and let $\beta\in(0,1)$. We define $\overline{\mathcal{F}}$ to be the family of maximal dyadic cubes $R$ that satisfy \[|R\cap(\cup_{Q\in{\mathcal{F}}}Q)|>\beta|R|\] and ${\mathcal{F}}_{1}$ to be the family of dyadic parents of the cubes from  $\overline{\mathcal{F}}$. Now let $E=\cup_{Q\in{\mathcal{F}}}Q$ and $E_1=\cup_{Q\in\mathcal{F}_1}Q$.
Then either
\[(i)\ \ |E_{1}|\ge \beta^{-1}|E|\quad {\rm{or}}\]
\[ \quad (ii)\ \ \beta^{-1}|E|>|Q_0|\ {\rm\ and}\ E_{1}=Q_0.\]
\end{lemma}

\begin{proof}
We prove this statement using induction on the rank of the smallest cube in $\mathcal{F}$.
If $Q_0\in\mathcal{F}$ then $(ii)$ holds. Otherwise we divide $\mathcal{F}$ into $2^n$ subfamilies
$\mathcal{F}^{(l)}=\{Q\in\mathcal{F}: Q\subset Q^{(l)}\}$.

We have one of the following cases:\\
(A) $Q_0\in\overline{\mathcal{F}}$,\\
(B) $Q_0\not\in\overline{\mathcal{F}}$ but $Q^{(l)}\in\overline{\mathcal{F}}$ for some $l$, or\\
(C) $Q_0\not\in\overline{\mathcal{F}}$ and $Q^{(l)}\not\in\overline{\mathcal{F}}$ for each $l$.

\medskip

\noindent For (A) we have $|\cup_{Q\in\mathcal{F}}Q|>\beta|Q_0|$ and $(ii)$ holds.

\medskip

\noindent For (B) we have $Q_0\in \mathcal{F}_1$  and $E_1=Q_0$. At the same time   $|Q_0|\ge\beta^{-1}|E|$ since $Q_0\not\in\overline{\mathcal{F}}$  and $(i)$ holds.

\medskip

\noindent For (C) we have by the induction hypothesis the statement is true for each family $\mathcal{F}^{(l)}$ and $E^{(l)}=\cup_{Q\in\mathcal{F}^{(l)}}Q$ if we replace $Q_0$ by $Q^{(l)}$. We see that $(ii)$ does not hold for $E^{(l)}$ since $Q^{(l)}\not\in\overline{\mathcal{F}}$ , i.e. $\beta^{-1}|E^{(l)}|\le|Q^{(l)}|$. Thus  $(i)$ holds for each $l$ and
\[|E_1|=\sum_l|E_1^{(l)}|\ge \beta^{-1}\sum_l|E^{(l)}|=\beta^{-1}|E|,\]
where $E_1^{(l)}=E_1\cup Q^{(l)}$.
\end{proof}
\begin{remark}
Note that \textit{(i)} may be true even if $E_1=Q_0$.
\end{remark}

\subsection{From a set of positive measure to a ball}
\label{s:mb} A cube $R$ is called $\delta$-good for a function
$v\in L^2(R)$ if
\[
\int_R|v|^2\le \delta|R|.\]

\begin{proposition}
\label{pr:4} Suppose that the coefficients of (\ref{eq:11}) satisfy (\ref{eq:00}-\ref{eq:01}).
Let $E$ be a compact measurable subset of
$\Omega(\rho_1)$, $\rho_1<\rho_0$, and let $u$ be a solution of
(\ref{eq:11}). Assume that $\|u\|^2_{L^2(E)}\le\epsilon^2|E|$.
Then there exists a cube $Q_0\subset\Omega$ with side length
$r_1=\kappa \rho_1$ and a finite set $\{Q_j\}$ of  dyadic
sub-cubes of $Q_0$ such that
\[|(\cup Q_j)\cap E|>c(\Omega,\rho_1)|E|\]
 and each $Q_j$ is $D\epsilon^2$-good for $u$, where
\beq \label{eq:41} D=a_nC(u,\Omega,\rho_1)^{\gamma_n} \eeq and
$a_n,\gamma_n$ depend only on the dimension $n$.
\end{proposition}

\begin{proof} We consider the set
\[E_1=\{x\in E: |u(x)|\le \sqrt{2}\epsilon\}.\]
Clearly, $|E_1|\ge |E|/2$. We may cover $E_1$ by finitely many
cubes with side-length $r_1$ and distance to the boundary of
$\Omega$ greater than $\rho_1$. We choose one of those $Q_0$ such
that $|E_1\cap Q_0|>c(\Omega,\rho_1)|E|$.

A dyadic sub-cube $K$ of $Q_0$ is called $\beta$-filled if
$|K\cap E_1|>\beta|K|$. Consider the set $\{Q_j\}$ of maximal
$\beta$-filled cubes (those $\beta$-filled cubes that are not
contained in any bigger $\beta$-filled cube). Since almost each
point of $E_1\cap Q_0$ is its point of density, we know that
$|(E_1\cap Q_0)\setminus\cup_j Q_j|=0$. Thus we can take finitely
many of cubes $\{Q_j\}$ such that
\[|\cup_j Q_j\cap E|>|\cup_j Q_j\cap E_1|>\frac{1}{2}|E\cap Q_0|>\frac{1}{2}c(\Omega,\rho_1)|E|\] and
$|Q_j\cap E_1|>\beta|Q_j|.$

Note that $\rho_1$ is fixed and let
$\alpha(u)=\alpha(u,\Omega,\rho_1)$, we can choose
$\beta=\beta(u)$ such that the last inequality implies
\[|B_j\cap E_1|>(1-\alpha)|B_j|,\]
for the ball $B_j$ inscribed in $Q_j$, i.e., $B_j$ has the same
center as $Q_j$ and radius $l_j/2$, where $l_j$ is the side length
of $Q_j$. Indeed
\[|B_j\cap E_1|\ge |B_j|-|Q_j\setminus E_1|\ge |B_j|-(1-\beta)|Q_j|=
|B_j|(1-(1-\beta)c_n).\] Thus $|B_j\cap E_1|\ge
(1-kc_nC(u,\Omega,\rho_1)^{-n})|B_j|\ge (1-\alpha)|B_j|$ if \beq \label{eq:43}
\beta=1-kC(u,\Omega,\rho_1)^{-n}, \eeq  where
$k<k_0$, $k_0$  does not depend on $u$ but depends on the
coefficients of the differential operator  and on $\Omega$. Then,
using (\ref{eq:20}) and (\ref{eq:28}), we get
\begin{multline*}
\int_{Q_j}|u|^2\le \int_{\sqrt{n}B_j}|u|^2\le
C(u,\Omega,\rho_1)^{1+\log n}\int_{B_j}|u|^2\le\\
2C(u,\Omega,\rho_1)^{1+\log n}\int_{B_j\cap E_1}|u|^2\le
2C(u,\Omega,\rho_1)^{1+\log n}|B_j\cap E_1|2\epsilon^2\le
D\epsilon^2|Q_j|.
\end{multline*}
\end{proof}

Our aim is to estimate $ \int_{Q_0}|u|^2,$ where $Q_0$
is the cube from the last proposition, so $u$ and $Q_0$ are fixed,
we consider dyadic sub-cubes of $Q_0$.
Let $D$ and $\beta<1$ be defined by (\ref{eq:41}) and (\ref{eq:43}).
We note that
\begin{itemize}
\item if $R$ is $\delta$-good then its dyadic parent is $D\delta$
good (by the doubling inequality (\ref{eq:20})), \item  if
$\{R_j\}$ are disjoint $\delta$-good cubes and $|R\cap(\cup
R_j)|>\beta|R|$ then $R$ is $D\delta$-good; this follows from
(\ref{eq:28}) and (\ref{eq:20}).
\end{itemize}

Let ${\mathcal{Q}}_1$ be the family of cubes $Q_j$ obtained in
Proposition \ref{pr:4}. We define by induction
$\overline{\mathcal{Q}}_j$ to be the family of maximal dyadic
cubes $R$ that satisfy
$|R\cap(\cup_{Q\in{\mathcal{Q}}_j}Q)|>\beta|R|$ and
${\mathcal{Q}}_{j+1}$ to be the family of dyadic parents of the
cubes from  $\overline{\mathcal{Q}}_j$.  Then by induction all
cubes from ${\mathcal{Q}}_j$ are $D^{2j-1}\epsilon^2$ good.

Let $E_j= \cup_{Q\in{\mathcal{Q}}_{j}}Q$. By Lemma \ref{l:squares} we have
\[{\rm{either}}\quad (i)\ \ |E_{j}|\ge \beta^{-j+1}|E_1|\quad {\rm{or}} \quad (ii)\ \ E_{j}=Q_0.\]
By Proposition \ref{pr:4}, $|E_1|\ge c(\Omega,\rho_1)|E|$.
Therefore, taking (\ref{eq:43}) into account, we see that there
exists \beq \label{eq:33} N=N(|E|,\rho_1,
P,V,W_1,W_2)C(u,\Omega,\rho_1)^n=N_0C(u,\Omega,\rho_1)^n \eeq such
that $E_N=Q_0$, where $N_0$ depends only on $|E|$ on $\rho_1$ and on
$\Omega$. Thus \beq\label{eq:44} \int_{Q_0}|u|^2\le
\left(a_nC(u,\Omega,\rho_1)^{\gamma_n}\right)^N\epsilon^2.\eeq

Now we can prove the following statement

\begin{proposition}
\label{pr:2} Assume that the equation (\ref{eq:11}) is given and its coefficients satisfy
(\ref{eq:00}-\ref{eq:01})  and
let $\sigma, m,\rho_1$ be positive, $\sigma<
1/n$. There exists
$\epsilon_0=\epsilon_0(\sigma,m,\rho_1)$ such that
the following holds:\\
If $E$ is a measurable subset of $\Omega(\rho_1)$, $|E|\ge m$, and
$u$ is a solution of (\ref{eq:11}) that satisfies
\[\|u\|_{L^2(E)}|E|^{-1/2}\le\epsilon<\epsilon_0\ {\text{and}}\ \|u\|_{L^2(\Omega)}\le 1\]
then there exists a ball $B_0$ of radius $r_1=\kappa\rho_1$ and
center in $\Omega(\rho_1)$ such that

\beq \label{eq:46} \int_{B_0}|u|^2\le
\left(|\log{\epsilon}|\right) ^{-2\sigma/{H_{1}}}, \eeq where
$H_{1}$ is as in Proposition \ref{pr:1}.
\end{proposition}
\begin{proof}
By the definition of  $C(u,\Omega,\rho_1)$, there exists a ball of
radius $r_1$ such that \beq \label{eq:45} \int_B|u|^2\le M^2\left(
(C_0C^{-1}(u,\Omega,\rho_1))\right)^{2/{H_{1}}},\eeq where
$M=\|u\|_{L^2(\Omega)}\le 1$.

If $C(u,\Omega,\rho_1)\ge C_0|\log\epsilon|^{\sigma}$ then
(\ref{eq:45}) implies the desired estimate. Otherwise we use the
ball $B_0$  inscribed in the cube $Q_0$ in (\ref{eq:44}) and from
(\ref{eq:33}) and (\ref{eq:20}) we obtain
\begin{multline*}
\int_{B_0}|u|^2\le \exp\left(2\log\epsilon+(N+1)\log (a_n C_o^{\gamma_n})+(N+1)\gamma_n\sigma\log|\log\epsilon|\right)\\
\le \exp\left(-2|\log\epsilon|+
2N_0C_0^n|\log\epsilon|^{n\sigma}(\log
(a_nC_0^{\gamma_n})+\gamma_n\sigma\log|\log\epsilon|)\right),
\end{multline*}
then (\ref{eq:46}) follows for $\epsilon$ small enough since
$n\sigma<1$.
\end{proof}
\begin{remark} The statement of Proposition implies also that there exists $A=A(\sigma, m, \rho_1)$ such that
if $E$ is a compact measurable subset of $\Omega(\rho_1)$,
$|E|>m$, and  $u$ is a solution of (\ref{eq:11}) that satisfies
$\|u\|_{L^2(E)}|E|^{-1/2}\le\epsilon<1$ and
$\|u\|_{L^2(\Omega)}\le 1$ then there exists a ball $B_0$ of
radius $r_1=\kappa\rho_1$ and center in $\Omega(\rho_1)$ such that
\[ 
\int_{B_0}|u|^2\le
A|\log{\epsilon}|)^{-2\sigma/H_1}. \]
\end{remark}

\subsection{Proof of the Main result}
In this section we prove Theorem \ref{th:m1} formulated in the
introduction. We use standard argument of smallness propagation, see \cite{ARRV}.
\begin{proof}
First, we may assume that $\rho<\rho_0$, then we cover
$\Omega(\rho)$ by finitely many balls of radii $r=\kappa\rho$ and
with centers in $\Omega(\rho)$. It is enough to prove a similar
inequality for $L^2$-norm of $u$ over each of those balls. Now we
refer to Proposition \ref{pr:2} to find one ball $B_0=B(x_0,r)$
with desired estimate and $c=\sigma/H_1$ and $C$ that depends on $|E|$
and $\rho$, we note that $r$ depends only on $\rho$. Using Theorem
\ref{pr:3} we can obtain an estimate for the norm of $u$ in
$2B_0=B(x_0,2r)$ and then in a ball of radius $r$ with center
$x_1$ such that $|x_0-x_1|=r$. By Theorem \ref{pr:3} we get
\[
\left\Vert u\right\Vert _{L^{2}(B_{1})}\leq C|\log\epsilon|^{-c}
.\]
Where $C$ and $c$ do not depend on $u$. 
\end{proof}
\begin{remark}
We choose to work with $L^2$-norms since the solutions we consider
include unbounded functions (see Section 2 of Introduction in
\cite{LU} for corresponding examples and general discussions). If
we assume that $V\in L^t(\Omega), t>n/2$, then $L^2$ inequalities
and elliptic estimates yield $L^\infty$-results (see \cite[chapter
III, \S 13]{LU}).
\end{remark}



\section{ On a theorem of Nadirashvili}

 We prove Theorem \ref{th:m2} in
this section. The proof follows the argument of N. Nadirashvili
\cite{N} (in the way we understand it). Our version of the proof
differs from the original in some technical details, it is
adjusted to our assumptions on coefficients.  For example we use
elliptic estimate in the place of the growth lemma of Landis, which
appeared in the original proof, we also apply the three sphere
inequality obtained in Section \ref{s:th} of this work.

\subsection{First reduction}
Once again we consider elliptic equations of the form
(\ref{eq:11}) such that the main term satisfies inequalities (\ref
{eq:00}), (\ref{eq:02}), and (\ref{eq:01n}) holds for the lower
order terms. The statement of the Theorem \ref{th:m2} follows from
the lemma below.

\begin{lemma}\label{l:N1}
Let $P=\di(g\nabla u)$, where $g(x)=\{g^{ij}(x)\}_{i,j=1}^{n}$ is
a real-valued symmetric matrix satisfying (\ref{eq:00}) and
(\ref{eq:02}). Let $s>n$ and assume that $\rho>0$ and $V$, $W_1$
and $W_2$ satisfy
\[
\|V\|_{L^{s/2}(\Omega)}, \|W_1\|_{L^{s}(\Omega)},\|W_2\|_{L^{s}(\Omega)}\le \Lambda_1.
\]
 Then there exist positive numbers $\delta$ and $c$ that depend on $\Omega, \lambda, \Lambda_0, \Lambda_1$ and $\rho$,
such that if
\beq \label{eq:11n}
Pu=Vu+W_1\cdot\nabla u+\nabla\cdot(W_2u), \quad {\rm{in}}\quad
\Omega,\eeq
$\|u\|_{L^\infty(\Omega)}\le 1$, $E$ is a measurable subset of
$B_{r/2}(x),\ |E|>(1-\delta)|B_{r/2}$, where $B_r(x)\subset\Omega(\rho), r\le 1$ and $\|u\|_{L^{\infty}(E)}\le \epsilon$,
$\epsilon\in(0,1/2)$ then
\[
\|u\|_{L^\infty(B_{r/2}(x))}\le \exp(-c|\log \epsilon|^{\alpha}),
\]
where $\alpha<1$ and depends on the dimension of the space, $\Omega, \lambda, \Lambda_0,$ and $\Lambda_1$.
\end{lemma}

We want to show that Lemma \ref{l:N1} implies Theorem \ref{th:m2}.
The three sphere inequality (\ref{eq:N1}) for $L^2$-norms
implies similar inequality for $L^\infty$ norms since
we assume that $V\in L^s(\Omega)$ and $s>n/2$. Then we obtain the following version of Lemma \ref{l:N1} for cubes:\\
{\it{There exist positive numbers $\delta_1$ and $c_1$  that
depend on $\Omega, \lambda, \Lambda_0, \Lambda_1$ and $\rho$ and $r_0$, such
that if (\ref{eq:11n}) holds, $\|u\|_{L^\infty(\Omega)}\le 1$, $E$
is a measurable subset of $Q_{r}(x),\ |E|>(1-\delta_1)|Q_{r}(x)$,
where $Q_{c_nr}(x)\subset\Omega(\rho), r\le r_0$ and
$\|u\|_{L^{\infty}(E)}\le \epsilon$, $\epsilon\in(0,1/2)$ then
\[
\|u\|_{L^\infty(Q_{r}(x))}\le \exp(-c_1|\log \epsilon|^{\alpha}),
\]
where $Q_t(x)$ is a cube with side length $t$ and center $x$.}}

Now we  find a cube $Q_0\subset \Omega(\rho)$ with side length $l=l(\Omega, \rho)$ and a finite collection of  its disjoint dyadic sub-cubes $Q_j=Q_{r_j}(x_j)$ such that
$|E\cap Q_j|>(1-\delta)|E|$ and $|E\cap(\cup_j Q_j|>a|E|$, where $a=a(\rho,\Omega,n)$. Using Lemma \ref{l:N1}, three sphere
inequality and Lemma \ref{l:squares} we conclude that there exists $c_0, C_0$ and $\mu_0$ that depend on
 $\Omega, \lambda, \Lambda_0, \Lambda_1, |E|$ and $\rho$,  such that
\[\|u\|_{L^\infty(Q_0)}\le C_0\exp(-c_0|\log\epsilon|^{\mu_0}).\]
By applying three sphere inequality once again we obtain (\ref{eq:50n}).

\subsection{Second reduction}
We shall formulate another statement that implies Lemma
\ref{l:N1}. Assume that Lemma is false, then for any $k\in\N
\cup\{0\}$ we can find $E_k\subset B_{r_k/2}(x_k)$ and $u_k$ such
that $u_k$ satisfies (\ref{eq:11n}),
$\|u_k\|_{L^\infty(\Omega)}\le 1$,
$|E_k|>(1-2^{-k})|B_{r_k/2}(x_k)|$, where
$B_{r_k/2}(x_k)\subset\Omega_\rho, r\le r_0$,
$\|u_k\|_{L^{\infty}(E)}\le \epsilon_k$, $\epsilon_k\in(0,1/2)$
and
\[
\|u_k\|_{L^{\infty}(B_{x_k}(r_k/2))}> \exp(-k|\log \epsilon_k|^{\alpha}).
\]

We consider $v_k(x)=u_k(x_k+rx)$ then $v_k$ satisfies an equation
of the form \beq \label{eq:inub}
\di(\tilde{g}v)=\tilde{V}v+\tilde{W}_1\cdot\nabla
v+\nabla\cdot(\tilde{W}_2 v)\quad {\rm{in}}\ \  B_1,\eeq where
$B_1$ is the unit ball with center at the origin. Moreover
$\tilde{g}$ satisfies (\ref{eq:00}) and (\ref{eq:02}) in $B_1$ and
\beq \label{eq:coefest} \|\tilde{V}\|_{L^{s/2}(B_1)},
\|\tilde{W}_1\|_{L^{s}(B_1)},\|\tilde{W}_2\|_{L^{s}(B_1)}\le
\Lambda_1. \eeq
Define $F_k=\{x: x_k+rx\in E_k\}\subset
B_{1/2}$, clearly $|F_k|>(1-2^{-k})|B_{1/2}|$. Let further
$F_0=\cap_{k\ge 2} F_k$, we have $|F_0|\ge \frac{1}{2}|B_{1/2}|$.
Assuming that Lemma \ref{l:N1} does not hold, we see that the
following statement should be false:

\begin{lemma}\label{l:N2}
Let $F_0$ be a subset of $B_{1/2}$ with
$|F_0|>\frac{1}{2}|B_{1/2}|$. There exist $c=c(F_0, \lambda,
\Lambda_0, \Lambda_1, s)$ such that if $u$ is a solution of
(\ref{eq:inub}), for which (\ref{eq:00}), (\ref{eq:02}), and
(\ref{eq:coefest}) holds,  $|u|\le 1$ in $B_1$, and $|u|\le
\epsilon$ on $F_0$, $\epsilon\in(0,1/2)$ then
\[
\|u\|_{L^\infty(B_{1/2})}\le \exp(-c|\log \epsilon|^{\alpha}).
\]where $\alpha<1$ and $c$ depend on the dimension of the space, $\Omega, \lambda, \Lambda_0,$ and $\Lambda_1$
\end{lemma}

Thus the argument of this subsection shows that it is enough to prove Lemma \ref{l:N2} and get an estimate with $c$ that depends on $F_0$, Lemma \ref{l:N1} will follow. For the rest of the proof $F_0$ is a fixed subset of $B_{1/2}$.

\subsection{Elliptic estimate}

The following elliptic estimate holds
\begin{equation}\label{ellipticestimate}
\max_{y\in B_r(x)}|u(y)|^2\le
\frac{A^2}{|B_{2r}(x)|}\int_{B_{2r}(x)}u^2,\quad {\rm where}\
B_{2r}(x)\subset B_1,
\end{equation}
 where
$A$ depends only on $n,\lambda,\Lambda_0,\Lambda_1$, see for
example \cite[chapter III, \S 13]{LU}. The next result follows
from the elliptic estimate and will be used repeatedly in the
sequel.

\begin{claim}
Let $u$ be a solution to (\ref{eq:inub}) in $B_1$, $F\subset B_{1/2}$
 and
$|u|<\epsilon$ on $F$. There exist $\gamma\in(0,1)$ and
$\beta>0$, $\beta$ depends only on $\gamma$ and on $A$, such that:\\
$|u(y^*)|>c>2A\epsilon$, $|F\cap B_{r^*}(x^*)|>\gamma
|B_{r^*}(x^*)|$, $|x^*|<1-4r^*$ and $y^*\in B_{r^*/2}(x^*)$ imply
\[
\sup_{B_{r^*}(x^*)} |u|>(1+\beta)c.
\]
\end{claim}

\begin{proof}
Denote $\max_{y\in B_{r^*}(x^*)}|u(y)|=d$, inequality
(\ref{ellipticestimate}) for $B_{r^*/2}(x^*)$ gives
\[
c^2<\max_{B_{x^*}(r^*/2)}|u|^2\le
\frac{A^2}{|B_{r^*}(x^*)|}\int_{B_{r^*}(x^*)}u^2\le
A^2(\epsilon^2\gamma+d^2(1-\gamma)).\] Then
\[
d^2>\frac{c^2-A^2\epsilon^2\gamma}{(1-\gamma)A^2}\ge\frac{c^2}{2(1-\gamma)A^2}.\]
If $\gamma$ is close to $1$, $\gamma=\gamma(A)$, then the claim is
justified.
 \end{proof}

\subsection{Points of density and Marcinkiewicz integral}

Let $\gamma$ be from the claim above, $\gamma\in(0,1)$. Since almost all the points of $F_0$ are
points of density, there exist a positive number $r_1$, $r_1$ depends on $F_0$, and a set
$F_1\subset F_0$ such that $|F_1|>|F_0|/2$ and for each $x\in F_1$ we have
\[
\frac{|F_0\cap B_r(x)|}{|B_r(x)|}>\gamma {\rm \ whenever\ } r\le
r_1.
\]
Now, $F_1$ has positive measure and by the
Marcinkiewicz theorem (see for example, \cite[chapter I]{S} )
for almost each point of $x$ of $F_1$ we
have
\begin{equation}\label{Marcinkiewicz}
\int_{|y|\le 1}\frac {\dist(x+y, F_1)}{|y|^{n+1}}<+\infty.
\end{equation}
We fix a point $x_0$ in $F_1$ for which (\ref{Marcinkiewicz})
holds.

For each $r\in (0, r_1)$ let $h(r)=\max_{|y|=r}\dist(x_0+y, F_1)$.
Our choice of $x_0$ implies that
\begin{equation}\label{h-Mar}
\int_0^{r_1}\frac{(h(r))^n}{r^{n+1}}dr<+\infty.
\end{equation}
Indeed, let us  check that (\ref{Marcinkiewicz}) implies (\ref{h-Mar}).
Let $\overline{y}$ be such that $|\overline{y}|=r$ and
$\dist(x_0+\overline{y}, F_1)=h(r)$. We have $\dist(x_0+z,F_1)\ge h(r)/2$ for all $z$ such that
$|x_0-z|=r$ and $|\overline{y}-z|<h(r)/2$. Then
\[
\int_{\s^n}\dist (x_0+ry', F_1)dy' \ge C
h(r)\left(\frac{h(r)}{r}\right)^{n-1}.\] Where
$\mathbf{S}^{n}=\left\{ y^{\prime }\in \mathbf{R}^{n}:\left\vert
y^{\prime }\right\vert =1\right\} $.
Finally, polar integration gives (\ref{h-Mar}).

Let further $h_l=\max_{r\in(2^{-l-1},2^{-l})}h(r)=h(r_l)$. We note that
$h(r_l+t)\ge h(r_l)-|t|>h_l/2$ when $|t|<h_l/2$. Then (\ref{h-Mar}) implies that
\begin{equation}
\label{hdiadic}
+\infty>\sum_{l>l_0}\int_{2^{-l-1}}^{2^{-l}}\frac{(h(r))^n}{r^{n+1}}\ge
\sum_{l>l_0} (h_l/2)^{n}2^{(n+1)l}h_l/2.
\end{equation}
The following property holds:\\
 for any $x$ such that
$|x-x_0|<r_1/2$ there is a ball $B_{r(x)}(s(x))$ such that $x\in
B_{r(x)/2}(s(x))$, \[r(x)\le 2h(|x-x_0|)\] and
\begin{equation}\label{elldouble}
\frac{|F_0\cap B_{r(x)}(s(x))|}{|B_{r(x)}(s(x))|}>\gamma.
\end{equation}
Indeed, we
just take $s(x)\in \overline{B_{h(|x-x_0|)}(x)}\cap F_1$ and
$r(x)=2|x-s(x)|$.

\subsection{Growth properties of $u$}
Let $r<r_1$ and $m(r)=\max_{|x-x_0|=r} |u(x)|$. Assume that
$m(r)>2A\epsilon$ and let $x, |x-x_0|=r,$ be such that
$m(r)=|u(x)|$, further let $s(x)$ and $r(x)$ be as in
(\ref{elldouble}). We apply the Claim to the ball $B_{r(x)}(s(x))$
and get
\[
\sup _{B_{r(x)}(s(x))}|u|>(1+\beta)m(r).\] We have also
$|s(x)-x_0|\le |x-x_0|+|x-s(x)|\le r+r(x)/2\le r+h(r)$ and
\[
\max_{t\in[-3h(r),3h(r)]}m(r+t)>(1+\beta)m(r).
\]
Let us define
\[r(M):=\min\{r:m(r)\ge M\}\]
for $M\le sup_{|x-x_0|\le r_1}|u(x)|=M_1$. For any $M>2A\epsilon$ we have either
\[(1+\beta)M>M_1,\quad  {\rm and}\quad r(M)>r_1/2\]
 (see inequality for $h(r)$ below) or
\begin{equation}
\label{rm}
r((1+\beta)M)\le r(M)+3h(r(M)).
\end{equation}

We remark that (\ref{hdiadic}) implies that $\lim_{l\rightarrow\infty}h_l2^l=0$ and there exists $l_1=l_1(F_0,x_0)$ such that $h_l2^{l}<1/12$  when $l>l_1$. Consequently,
$h(r)<r/6$ for $r<2^{-l_1}$.

We say
that $l$ is good (and the corresponding interval
$(2^{-l-1},2^{-l})$ is good) if $h_l<l^{-1/(n+1)}2^{-l}$. Then
(\ref{hdiadic}) implies that there exists $N_0$, $N_0=N_0(F_0,x_0)$, such that for
$N\ge N_0$ at least $2^{N-1}$ of the numbers $2^{N}+1$,...,$2^{N+1}$
are good.

Assume that  $r_0=r(2A\epsilon)$ we want to prove that
\beq
\label{eq:r0}
r_0\ge \exp(-B|\log\epsilon|^{(n+1)/(n+2)}),\eeq
where $B$ depends on $F_0, x_0, A$.
We assume also that $r_0\in(2^{-l_0-1},2^{-l_0})$, where $l_0>l_1$ and $l_0\in(2^{N+1},2^{N+2}),$ $N\ge N_0$. (Otherwise (\ref{eq:r0}) is satisfied provided that $\epsilon<1/2$ and $B$ is large enough.)

Further let
$r_j=r((1+\beta)^j2A\epsilon)$, when $(1+\beta)^{j}2A\epsilon<M_1$. Then (\ref{rm}) implies
$r_{j+1}\le r_j+3h(r_j).$ The sequence $\{r_j\}$ is increasing and
 $r_{j+1}-r_j<3h(r_j)<r_j/2;$ moreover  if $r_j\in(2^{-l-1},2^{-l}]$,
where $l$ is good then
\beq
\label{eq:lgood}
r_{j+1}-r_j<3h_l<2^{-l+2}l^{-1/(n+1)}.\eeq
Let $K$ be the number of $j$ such that $r_j<2^{-2^N}$. For each good $l$, $2^N<l\le 2^{N+1}$,
there exists $j_l=\min\{j: r_j\in(2^{-l-1}, 2^{-l}]\}$.
Thus $r_{j_l-1}<2^{-l-1}$ and $r_{j_l}<\frac{3}{2}r_{j_l-1}<3\cdot2^{-l-2}$.
Then (\ref{eq:lgood}) implies that there are at least $\frac{1}{4}l^{1/(n+1)}$ elements of the sequence $\{r_j\}$ in $(2^{-l-1}, 2^{-l})$.
Now, since there are at least $2^{N-1}$ good numbers $l\in\{2^N+1,...,2^{N+1}\}$, we have
\[K\ge  2^{N-1}\cdot\frac 14 2^{N/(n+1)}=\frac 1 8  2^{N(n+2)/(n+1)}.\]

From the other hand $m(r_K)\ge (1+\beta)^K2A\epsilon$ and $m(r_K)\le1$. We
get the following inequality:
\[
2A\epsilon(1+\beta)^K\epsilon\le 1.\] It implies
$
K\le a|\log \epsilon|,$
where $a$ depends on $A$ and on $\beta$. We combine the last inequality with the estimate
we have for $K$ from below and obtain
\[
8a|\log\epsilon|\ge  2^{N\frac{n+2}{n+1}}.
\]
Now,
\[r_0\ge 2^{-2^N}=\exp(-2^N\log 2)\ge \exp(-B|\log\epsilon|^{(n+1)/(n+2)}).\]
Inequality (\ref{eq:r0}) is established.

We have
\[
\max_{B_{r_0}(x_0)} |u |\le 2A\epsilon.\] Now we apply Theorem \ref{pr:3}  and obtain
\[
\|u\|_{L^2(B_{\kappa/2}(x_0))}\le C\exp(-B_1|\log\epsilon|^{1/n+2}).\]
Finally, using standard technique, we complete the proof of Lemma
\ref{l:N2}. We note that $\alpha$ depends on $n$ and $\kappa$ from Proposition \ref{pr:1}.

\section*{Acknowledgements}
We are very grateful to Professor   H.Koch for useful comments and suggestions that improved the presentation and results of the current manuscript.

The work on this article started when the first author visited Universit\`{a}
degli Studi, Firenze,  and it's a pleasure to thank the University for its hospitality and support.
The first author is partly supported by the Research
Council of Norway, grants 160192/V30 and  177355/V30.


\end{document}